\documentclass[a4paper,leqno]{amsart}                

\usepackage[latin1]{inputenc}                  
\usepackage[T1]{fontenc}                       
\usepackage[spanish,english]{babel}            
\usepackage{amsmath,amssymb,amsthm}            
\usepackage{latexsym}                          
\usepackage{delarray}                          
\usepackage{bbm}                               
\usepackage{hyperref}                          
\usepackage[pdftex,usenames,dvipsnames]{color} 
\usepackage{bbding,trfsigns}                   
\usepackage{wasysym}                           

\newtheorem{Th}{Theorem}[section]              
\newtheorem{Cor}{Corollary}[section]
\newtheorem{Prop}[Th]{Proposition}
\newtheorem{Lem}{Lemma}[section]

\theoremstyle{definition}
\newtheorem{Rm}{Remark}[section]

\newcommand{\C}{\mathbb{C}}
\newcommand{\N}{\mathbb{N}}
\newcommand{\Z}{\mathbb{Z}}
\newcommand{\R}{{\mathbb{R}}}

\newcommand{\T}{{\mathbb{T}}}

\newcommand{\eps}{{\varepsilon}}

\DeclareMathAlphabet{\mathpzc}{OT1}{pzc}{m}{it} 

\DeclareMathOperator{\sgn}{sgn}
\DeclareMathOperator{\supp}{supp}

\DeclareFontFamily{U}{mathx}{\hyphenchar\font45}
\DeclareFontShape{U}{mathx}{m}{n}{
      <5> <6> <7> <8> <9> <10>
      <10.95> <12> <14.4> <17.28> <20.74> <24.88>
      mathx10
      }{}
\DeclareSymbolFont{mathx}{U}{mathx}{m}{n}
\DeclareFontSubstitution{U}{mathx}{m}{n}
\DeclareMathAccent{\widecheck}{0}{mathx}{"71}
\DeclareMathAccent{\wideparen}{0}{mathx}{"75}

\title[Bounds for partial derivatives]
      {Bounds for partial derivatives:\\ necessity of UMD and sharp constants}

\author[A.J. Castro]{Alejandro J. Castro}
\author[T. P. Hyt\"onen]{Tuomas P. Hyt\"onen}

\address{
        A.J. Castro,
        Departamento de Análisis Matemático,
        Universidad de la Laguna,
        Campus de Anchieta, Avda. Astrofísico Francisco Sánchez, s/n,
        38271, La Laguna (Sta. Cruz de Tenerife), Spain}
\email{ajcastro@ull.es}

\address{
        T.P. Hyt\"onen,
        Department of Mathematics and Statistics, University of Helsinki,
        Gustaf H\"allstr\"omin katu 2b, FI-00014 Helsinki, Finland}
\email{tuomas.hytonen@helsinki.fi}


\makeatletter
\@namedef{subjclassname@2010}{%
  \textup{2010} Mathematics Subject Classification}
\makeatother


\thanks{T.P.H. is supported by the European Union via the ERC Starting Grant
``Analytic--probabilistic methods for borderline singular
integrals''.
A.J.C is supported by a FPU grant from Spanish Government.
This research was mostly carried out during A.J.C.'s
visit to the University of Helsinki in Autumn 2013.}

\begin{document}


  \maketitle                

\begin{abstract}
We prove the necessity of the UMD condition, with a quantitative estimate of the UMD constant, for any inequality in a family of $L^p$ bounds between different partial derivatives $\partial^\beta u$ of $u\in C^\infty_c(\R^n,X)$. In particular, we show that the estimate $\|u_{xy}\|_p\leq K(\|u_{xx}\|_p+\|u_{yy}\|_p)$ characterizes the UMD property, and the best constant $K$ is equal to one half of the UMD constant. This precise value of $K$ seems to be new even for scalar-valued functions.
\end{abstract}

\section{Introduction} \label{sec:intro}

A priori estimates between different partial derivatives such as
\begin{equation}\label{eq:estim}
    \|\partial^\beta u\|_{L^{q}(\R^n,X)}
        \leq K \sum_{j=1}^N \|\partial^{\alpha^j} u\|_{L^{p_j}(\R^n,X)}, \quad u \in C_c^\infty(\R^n,X),
\end{equation}
play an important role in Analysis. Here $\partial^\beta=\partial_1^{\beta_1} \cdots \partial_n^{\beta_n}$ is the usual derivative associated to the multi-index
$\beta=(\beta_1, \dots, \beta_n)$. In the case of scalar-valued functions, $X=\C$ or $X=\R$, and exponents in the range $q,p_j\in(1,\infty)$, a complete characterization of admissible sets of $q,p_j$ and $\beta,\alpha^j$ is a classical result of O. Besov, V. Il$'$in and S. Nikol$'$ski{\u\i} \cite{BNS}. The second author showed in \cite{Hy4} that the same characterization remains valid for vector-valued functions, provided that the target space $X$ has the UMD property.

A Banach space $X$ has this property
if, for some (equivalently, for all) $1<p<\infty$, there exists a constant $C_p<\infty$, such that
\begin{equation}\label{eq:defUMD}
    \Big\| \sum_{\ell=1}^r \sigma_\ell d_\ell \Big\|_{L^p(\Omega,X)}
        \leq C_p \Big\| \sum_{\ell=1}^r d_\ell \Big\|_{L^p(\Omega,X)},
\end{equation}
for each $r \in \N$, every sequence of signs $(\sigma_\ell)_{\ell \in \N}$ in $\{-1,1\}$
and any martingale difference sequence
$(d_\ell)_{\ell \in \N}$ in $L^p(\Omega,X)$. Here
$L^p(\Omega,X)$ represents the space of all functions
that take values in the Banach space $X$ and their $X$ norm is
$p$-integrable on $\Omega$. We denote by $\beta_p(X)$ the smallest admissible constant $C_p$.
Main properties of UMD spaces can be found in the survey of J. L. Rubio de Francia \cite{Rub}.

Our present goals are two-fold. On the one hand, we show that the result of \cite{Hy4} is optimal in the sense that it establishes the most general class of Banach spaces  where the classical Besov--Il'in--Nikol'ski{\u\i} result \cite{BNS} can be generalized; namely, we exhibit instances of \eqref{eq:estim}, for which UMD is not only sufficient (as shown in \cite{Hy4}) but also necessary. This continues a tradition of diverse characterizations of the UMD property (see for instance \cite{Bou3,Bu2,CL,GMS,Gue,HTV,Hy3,Hy,KaWe,KuWe,Xu}, amongst others). On the other hand, as a byproduct, we obtain information on the size of the constant $K$ in \eqref{eq:estim}, which seems to be new already in the scalar case. Recall that the list of analytic inequalities, where the sharp constant can be determined, is somewhat restricted, whereas the UMD inequality \eqref{eq:defUMD} for $X=\C$ or $X=\R$ is one of the prominent positive examples with $\beta_p(\C)=\beta_p(\R)=(p^*-1)$; here $p^*=\!
 max\{p,p'\}$. This is a celebrated result of Burkholder; see \cite{Bu5}, \cite[p. 12]{Bu4} and \cite[Theorem 14]{Bu6}. Thus it is useful to relate other estimates to \eqref{eq:defUMD}.

A particular case of \eqref{eq:estim} that we have in mind is the basic bound
\begin{equation}\label{eq:mainExample}
  \|\partial_1\partial_2 u\|_{L^p(\R^2,X)}
  \leq K(\|\partial_1^2 u\|_{L^p(\R^2,X)}+\|\partial_2^2 u\|_{L^p(\R^2,X)}).
\end{equation}
Since $\partial_1\partial_2 u=R_1R_2\triangle u$, where $R_i$ is the Riesz transform in the $i$th direction, it follows at once that
\begin{equation}\label{eq:R12comput}
\begin{split}
  \|\partial_1\partial_2 u\|_{L^p(\R^2,X)}
  &\leq\|R_1R_2\|_{\mathcal{L}(L^p(\R^2,X))}\|\partial_1^2 u+\partial_2^2 u\|_{L^p(\R^2,X)} \\
   &\leq\|R_1R_2\|_{\mathcal{L}(L^p(\R^2,X))}(\|\partial_1^2 u\|_{L^p(\R^2,X)}+\|\partial_2^2 u\|_{L^p(\R^2,X)}),
\end{split}
\end{equation}
and $\|R_1R_2\|_{\mathcal{L}(L^p(\R^2,X))}=\beta_p(X)/2$ by \cite{GMS}, so that $K\leq \beta_p(X)/2$ in \eqref{eq:mainExample}. We prove that this constant, optimal for the first line in \eqref{eq:R12comput}, is also sharp for the seemingly weaker inequality \eqref{eq:mainExample}.

It is interesting that the best constant in such a simple-looking classical inequality appears to be previously unknown.
This constant is somewhat connected to the famous open problem of determining the $L^p$ norm of the Beurling transform, whose imaginary part is $2R_1R_2$; see \cite{GMS} for more details. It has been long conjectured \cite{Iwa} that the norm of the Beurling operator on $L^p(\R^2)$ is $\beta_p(\C)$, and a surprising contribution of \cite{GMS} was to show that this conjectured norm is already achieved by the imaginary part alone. Our result further amplifies this phenomenon.

We then turn to the precise formulation of our main results. We are only concerned with the case that $p_j\equiv q=:p$ in \eqref{eq:estim}, and all multi-indices $\beta,\alpha^j$ have equal length, which we denote by $|\beta|=\beta_1 + \dots +\beta_n$, when $\beta=(\beta_1, \dots, \beta_n) \in \N^n$. We also need to assume that the components of $\beta$ have sufficiently different parity from those of each $\alpha^j$.

\begin{Th}\label{Th:main}
Let $1<p<\infty$, $N \in \N$ and $\alpha^j$, $\beta \in \N^n$. Let $|\beta|=|\alpha^j|$ for all $j=1, \ldots, N$,
and suppose that there exists a set $F\subset\{1,\ldots,n\}$ such that $\sum_{\ell\in F}\alpha^j_\ell$ has the same parity (even or odd) for each $j=1,\ldots,N$, which is different from the parity of $\sum_{\ell\in F}\beta_{\ell}$. If the following estimate holds:
   \begin{equation}\label{eq:estim2}
    \|\partial^\beta u\|_{L^p(\R^n,X)}
        \leq K \sum_{j=1}^N \|\partial^{\alpha^j} u\|_{L^p(\R^n,X)}, \quad u \in C_c^\infty(\R^n,X),
    \end{equation}
then $X$ is UMD and $\beta_p(X) \leq KN$.
\end{Th}

\begin{Rm}
It is known (see e.g. \cite{Hy4}) that a necessary condition for \eqref{eq:estim2} to hold in any Banach space is that $\beta$ is a convex combination of the multi-indices $\alpha^j$. Although we will not explicitly use this fact, it is worth observing that the assumption \eqref{eq:estim2} is only meaningful for such sets of multi-indices.
\end{Rm}

The following corollaries demonstrate the somewhat technical conditions on the multi-indices assumed in the theorem:

\begin{Cor}
Let $X$ be a Banach space and $1<p<\infty$. Let $\beta\in\N^n$ be a multi-index such that $|\beta|$ is even, but $\beta_s$ is odd for at least one index $s\in\{1,\ldots,n\}$. Then the estimate
\begin{equation*}
  \|\partial^\beta u\|_{L^p(\R^n,X)}\leq K\sum_{j=1}^n\|\partial_j^{|\beta|}u\|_{L^p(\R^n,X)},\qquad u\in C^\infty_c(\R^n,X),
\end{equation*}
holds if and only if $X$ is a UMD space, and in this case $\beta_p(X)\leq Kn$.
\end{Cor}

\begin{proof}
The ``if'' part is well known; for instance, it is a special case of the results of \cite{Hy4}.
For the ``only if'' part, we apply Theorem~\ref{Th:main} with $\alpha^j=|\beta|e_j$ and $F=\{s\}$.
Then $\sum_{\ell\in F}\beta_\ell=\beta_s$ is odd, whereas $\sum_{\ell\in F}\alpha^j_\ell=|\beta|\delta_{js}$
is even for all $j$, as both $0$ and $|\beta|$ are even. Thus the Theorem implies that $X$ is UMD, with the
asserted estimate for the constant.
\end{proof}

\begin{Cor}\label{Cor:bestconst1}
    Let $X$ be a Banach space and $1<p<\infty$. Then \eqref{eq:mainExample} holds if and only if $X$ is a UMD space, and the best constant $K$ satisfies $K=\beta_p(X)/2$.
\end{Cor}

\begin{proof}
We already explained that $K\leq\beta_p(X)/2$ in \eqref{eq:R12comput}; this part of the result is not new.
The converse estimate is a special case of the previous Corollary with $n=2$ and $\beta=(1,1)$.
\end{proof}

There are many characterizations of the UMD property by estimates between $L^p$ norms of two different objects. A novel qualitative feature of the above characterizing conditions is that they involve a sum of several $L^p$ norms on the right hand side.

The proof of Theorem~\ref{Th:main} proceeds via the theory of Fourier multipliers.
We denote by $\widehat{f}$ the Fourier transform of the function $f$, defined as usual by
$$\widehat{f}(\xi)
    = \int_{\R^n} e^{-2 \pi i \xi \cdot x} f(x) dx, \quad \xi \in \R^n,$$
and by $\widecheck{f}$ its inverse, given as above but without the plus
sign in the complex exponential. The Fourier multiplier of a function $m$ is
defined by $T_m f = (m \widehat{f})^{\vee}$.

We write
\begin{equation}\label{eq:mmj}
    m(\xi)= \frac{\xi^\beta}{|\xi|^{|\beta|}}
        \quad \text{ and } \quad
    m_j(\xi)= \frac{\xi^{\alpha^j}}{|\xi|^{|\alpha^j|}}, \qquad \beta, \alpha^j \in \N^n, \ j=1, \dots, N,
\end{equation}
where $\xi^\beta=\xi_1^{\beta_1} \cdots \xi_n^{\beta_n}$.
By well-known properties of the Fourier transform we have that
$$\partial^\beta u = T_m (\Delta^{|\beta|/2}u),$$
where $\Delta=-\partial_1^2 - \dots - \partial_n^2$ is the positive Laplacian in $\R^n$
and $\Delta^{|\beta|/2}u=[(2\pi i|\xi|)^{|\beta|} \widehat{u}]^{\vee}$.
Hence, \eqref{eq:estim2} with $|\alpha^j|=|\beta|$,
$j=1, \dots N$, implies that
\begin{equation}\label{eq:estim3}
\|T_m v\|_{L^p(\R^n,X)}
        \leq K \sum_{j=1}^N \| T_{m_j}v\|_{L^p(\R^n,X)}, \quad v \in \mathcal{S}_0(\R^n,X),
\end{equation}
where $\mathcal{S}_0(\R^n,X)$ is the subspace of the Schwartz class constituted by functions having zero mean value.
Moreover, in order to prove that $X$ is UMD it is more convenient to have an estimate of type \eqref{eq:estim3},
but in terms of discrete Fourier multipliers (see Propositions~\ref{Prop:RnTn} and \ref{Prop:TnTrn} below).

This paper is organized as follows. In Section~\ref{sec:transf} we collect some transference results about
Fourier multipliers and in Section~\ref{sec:proofTh} we present the proof of Theorem~\ref{Th:main}.

\section{Transference results for Fourier multipliers} \label{sec:transf}

Throughout this section we assume that $m$ is a measurable bounded function verifying that
\begin{equation}\label{eq:mk}
    m(k)=\lim_{r \to 0} \frac{1}{|B(k,r)|} \int_{B(k,r)}m(\xi)d\xi,
\end{equation}
exists for every $k \in \Z^n$. The discrete Fourier multiplier
$T_{\widetilde{m}}f$, of a function $f$ defined in the $n$-dimensional torus $\T^n=[-1/2,1/2)^n$, is given by
\begin{equation}\label{eq:Tmdiscret}
    T_{\widetilde{m}}f (t)
        = \sum_{k \in \Z^n} m(k) \widehat{f}(k) e_k(t), \quad t \in \T^n,
\end{equation}
where $e_k(t)=e^{2\pi i t \cdot k}$ and $\widehat{f}(k)=\int_{\T^n} e_k(-t) f(t)dt$.

Next, we present some lemmas which will be very useful in the proof of
Proposition~\ref{Prop:RnTn}. We say that $f$ is a trigonometric polynomial in $X$ if
$$f(t)
    = \sum_{k \in \Z^n} a_k e_k(t), \quad t \in \T^n,$$
with finitely many nonzero coefficients $a_k \in X$. Note that $a_k=\widehat{f}(k)$.

\begin{Lem}\label{Lem:11.20}
    Let $\phi$, $\psi \in \mathcal{S}(\R^n)$ be radial functions satisfying
    $\int_{\R^n} \phi \psi d\xi=1 $.
    Then,
    $$\int_{\T^n}\langle g(t), T_{\widetilde{m}}f(t) \rangle dt
        = \lim_{\eps \to 0} \eps^n \int_{\R^n}\langle \psi(\eps x)g(x) , T_m(\phi(\eps \cdot )f)(x)\rangle dx,$$
    for every trigonometric polynomials $f:\T^n \longrightarrow X$, $g:\T^n \longrightarrow X^*$.
    Here $\langle \cdot , \cdot \rangle$ denotes the pairing between $X$ and its dual $X^*$.
\end{Lem}

\begin{proof}
    The key step in this proof is the identity
    \begin{equation}\label{eq:11.21}
        \lim_{\eps \to 0} \frac{1}{\eps^n}
            \int_{\R^n} m(\xi) \widehat{\phi}\Big(\frac{\xi - k}{\eps} \Big) \widecheck{\psi}\Big(\frac{\xi - \ell}{\eps} \Big) d\xi
            = \chi_{\{k=\ell\}} m(k), \quad k, \ell \in \Z^n,
    \end{equation}

    Assume first that $k=\ell$. Since the Fourier transform (and its inverse) of a radial function, is also radial
    (see for example \cite[p. 430]{Ste2}), we consider
    $\rho(|\xi|)=\widehat{\phi}(\xi) \widecheck{\psi}(\xi)$, which satisfies $\int_{\R^n}\rho(|\xi|)d\xi=\int_{\R^n}\widehat\phi\widecheck\psi\,d\xi=\int_{\R^n}\phi\psi\,dx=1$. We have that
    \begin{align*}
        & \frac{1}{\eps^n} \int_{\R^n} m(\xi) \widehat{\phi}\Big(\frac{\xi - k}{\eps} \Big) \widecheck{\psi}\Big(\frac{\xi - k}{\eps} \Big) d\xi
            = \int_{\R^n} m(\eps \eta + k) \rho(|\eta|) d\eta
            = - \int_{\R^n} m(\eps \eta + k) \int_{|\eta|}^\infty \rho'(r) dr d\eta \\
        & \qquad = - \int_0^\infty \rho'(r) \int_{B(0,r)} m(\eps \eta + k)  d\eta dr
            = - \int_0^\infty \rho'(r) \frac{|B(0,r)|}{|B(k,\eps r)|}\int_{B(k,\eps r)} m(\xi)  d\xi dr.
    \end{align*}
    By taking into account $\rho \in \mathcal{S}(\R)$, $m \in L^\infty(R^n)$ and \eqref{eq:mk}, we can apply
    dominated convergence and then integrate by parts to obtain
    \begin{align*}
        & - \lim_{\eps \to 0} \int_0^\infty \rho'(r) \frac{|B(0,r)|}{|B(k,\eps r)|}\int_{B(k,\eps r)} m(\xi)  d\xi dr
            = - m(k) |B(0,1)| \int_0^\infty  r^n \rho'(r)  dr \\
        & \qquad =  m(k) |B(0,1)| \int_0^\infty n r^{n-1} \rho(r) dr
            =  m(k)  \int_0^\infty \int_{\partial B(0,r)} \rho(r) d\sigma  dr
            =  m(k)  \int_{\R^n} \rho(|\xi|) d\xi
            = m(k),
    \end{align*}
    because $|\partial B(0,r)|=|B(0,1)| n r^{n-1}$.

    Suppose now that $k \neq \ell$. This time,
    $$\frac{1}{\eps^n} \int_{\R^n} m(\xi) \widehat{\phi}\Big(\frac{\xi - k}{\eps} \Big) \widecheck{\psi}\Big(\frac{\xi - \ell}{\eps} \Big) d\xi
        = \int_{\R^n} m(\eps \eta + k) \widehat{\phi}(\eta) \widecheck{\psi}\Big(\eta + \frac{k - \ell}{\eps} \Big) d\eta,$$
    where the integrand is bounded by $\|m\|_\infty \|\widecheck\psi\|_\infty |\widehat\phi(\eta)| \in L^1(d\eta)$,
    and converges pointwise to zero as $\eps \to 0$, since $\widecheck{\psi}(\eta + (k - \ell)/\eps)\to 0$.
    Applying again dominated convergence
    we readily see that the integral converges to zero as $\eps \to 0$.

    Take the  trigonometric polynomials $f=\sum_{k \in \Z^n} a_k e_k$ and
    $g=\sum_{\ell \in \Z^n} b_\ell e_\ell$ with $a_k \in X$ and $b_\ell \in X^*$.
    An application of \eqref{eq:11.21} and making use of the usual properties of the Fourier transform, we deduce
    \begin{align*}
        & \int_{\T^n}\langle g(t), T_{\widetilde{m}}f(t) \rangle dt
            = \sum_{\ell, k \in \Z^n}  \langle b_\ell , a_k \rangle m(k) \int_{\T^n} e_\ell(t) e_k(t) dt
            = \sum_{\ell, k \in \Z^n}  \langle b_{-\ell} , a_k \rangle  \chi_{\{k=\ell\}} m(k) \\
        & \qquad = \lim_{\eps \to 0} \eps^n \int_{\R^n}
            \Big \langle \sum_{\ell\in \Z^n} b_{-\ell} \frac{1}{\eps^n} \widecheck{\psi}\Big(\frac{\xi-\ell}{\eps}\Big),
                m(\xi) \sum_{k \in \Z^n} a_k \frac{1}{\eps^n} \widehat{\phi}\Big(\frac{\xi-k}{\eps}\Big) \Big \rangle d\xi\\
        & \qquad = \lim_{\eps \to 0} \eps^n \int_{\R^n}
            \Big \langle \Big[ \psi(\eps \cdot) g \Big]^{\vee}(\xi),
                m(\xi) \Big[  \phi(\eps \cdot) f \Big]^{\wedge}(\xi) \Big \rangle d\xi
        = \lim_{\eps \to 0} \eps^n \int_{\R^n}\langle \psi(\eps x)g(x) , T_m(\phi(\eps \cdot )f)(x)\rangle dx.
    \end{align*}
\end{proof}

\begin{Lem}\label{Lem:11.21}
    Let $1<p<\infty$. Then, for every $\phi \in \mathcal{S}(\R^n)$ and $f \in L^p(\T^n,X)$, we have that
    $$\lim_{\eps \to 0} \eps^{n/p} \| \phi(\eps \cdot) f \|_{L^p(\R^n,X)}
        = \| \phi \|_{L^p(\R^n)} \| f \|_{L^p(\T^n,X)}.$$
\end{Lem}

\begin{proof}
    The periodicity of the function $f$ allow us to write
    \begin{align*}
        \eps^n \int_{\R^n} \| \phi(\eps x) f(x)\|_X^p dx
        & = \eps^n \sum_{k \in \Z^n} \int_{\T^n+k} |\phi(\eps x)|^p \|f(x)\|_X^p dx \\
        & = \int_{\T^n} \Big(\eps^n \sum_{k \in \Z^n} |\phi(\eps (x-k))|^p \Big) \|f(x)\|_X^p dx,
    \end{align*}
    where the quantity in parentheses is a Riemann sum of the function $|\phi(\cdot)|^p$,
    which is uniformly bounded in $x$ and $\eps$. Now, this lemma is a simple consequence of the
    dominated convergence theorem.
\end{proof}

\begin{Lem}\label{Lem:11.23}
    Let $1<p<\infty$ and $p'=p/(p-1)$. Then, for each $f \in \mathcal{S}(\R^n,X)$, we have that
    $$\lim_{\eps \to 0} \eps^{n/p'} \Big\| \sum_{k \in \Z^n} \widehat{f}(\eps k) e_k \Big\|_{L^p(\T^n,X)}
        =  \| f \|_{L^p(\R^n,X)}.$$
\end{Lem}

\begin{proof}
Recall that we identify $\T^n=[-1/2,1/2)^n$; whence the notation
$\eps^{-1}\T^n=[-1/(2\eps),1/(2\eps))^n$ is meaningful below.
    Define $f_\eps=\eps^{-n}f(\cdot/\eps)$. The Poisson summation formula (see \cite[Theorem 3.1.17]{Graf}),
    tells us that
    $$\sum_{k \in \Z^n} \widehat{f}(\eps k) e_k(t)
        = \sum_{k \in \Z^n} \widehat{f_\eps}(k) e_k(t)
        = \sum_{k \in \Z^n} f_\eps(t+k)
        = \frac{1}{\eps^n}\sum_{k \in \Z^n} f\Big(\frac{t}{\eps}+\frac{k}{\eps}\Big), \quad t \in \T^n.$$
    Then, making the change of variables $u=t/\eps$,
    \begin{align*}
        \eps^{n/p'} \Big\| \sum_{k \in \Z^n} \widehat{f}(\eps k) e_k \Big\|_{L^p(\T^n,X)}
            & = \eps^{-n/p} \Big\| \sum_{k \in \Z^n} f\Big(\frac{\cdot}{\eps}+\frac{k}{\eps}\Big) \Big\|_{L^p(\T^n,X)}
              = \Big\| \sum_{k \in \Z^n} f\Big(\cdot +\frac{k}{\eps}\Big) \Big\|_{L^p(\eps^{-1}\T^n,X)} \\
            & =\|f\|_{L^p(\eps^{-1}\T^n,X)}
                + \mathcal{O}\Big( \sum_{k \in \Z^n \setminus \{0\}} \Big\|  f\Big(\cdot +\frac{k}{\eps}\Big) \Big\|_{L^p(\eps^{-1}\T^n,X)} \Big).
    \end{align*}
    It is obvious that
    $$\|f\|_{L^p(\eps^{-1}\T^n,X)} \longrightarrow \| f \|_{L^p(\R^n,X)}, \quad  \text{as } \eps \to 0.$$
    We analyze the second term. Since $f \in \mathcal{S}(\R^n,X)$, for every $N \in \N$, we have that
    $$\Big\| f\Big(u +\frac{k}{\eps} \Big)\Big\|_X
        \leq C \Big(\frac{\eps}{|k|}\Big)^{N}, \quad u \in \eps^{-1}\T^n.$$
    Hence, it is enough to take $N>n$ to see that
    \begin{equation*}
\sum_{k \in \Z^n \setminus \{0\}} \Big\|  f\Big(\cdot +\frac{k}{\eps}\Big) \Big\|_{L^p(\eps^{-1}\T^n,X)}
        \leq C \eps^{N-n/p} \sum_{k \in \Z^n \setminus \{0\}} \frac{1}{|k|^N}
        \leq C \eps^{N-n/p} \longrightarrow 0, \quad \text{as } \eps \to 0.\qedhere
\end{equation*}
\end{proof}

We now establish our transference result between $\R^n$ and $\T^n$.
We say that $m$ is a homogeneous function (of order zero) when
$m(\lambda \xi)=m(\xi)$, $\xi \in \R^n$, $\lambda>0$. We denote
$$L^p_0(\T^n,X)=\Big\{f \in L^p(\T^n,X) : \int_{\T^n}f dt = 0\Big\}.$$

\begin{Prop}\label{Prop:RnTn}
    Let $1<p<\infty$, $N \in \N$ and $m$, $m_j \in C^\infty(\R^n \setminus \{0\})$, $j=1, \dots, N$,
    be homogeneous functions.  Then, the following assertions are equivalent:
    \begin{itemize}
        \item[$(i)$] $\displaystyle \|T_m f\|_{L^p(\R^n,X)}
                            \leq K \sum_{j=1}^N \| T_{m_j}f\|_{L^p(\R^n,X)}, \quad f \in L^p(\R^n,X),$
        \item[$(ii)$] $\displaystyle \|T_{\widetilde{m}} f\|_{L^p(\T^n,X)}
                            \leq K \sum_{j=1}^N \| T_{\widetilde{m}_j}f\|_{L^p(\T^n,X)}, \quad f \in L^p_0(\T^n,X).$
    \end{itemize}
\end{Prop}

\begin{proof}[Proof of Proposition~\ref{Prop:RnTn}, $(ii) \Rightarrow (i)$]
    Let $f \in \mathcal{S}(\R^n,X)$ be such that $\supp \widehat{f} \subset \R^n \setminus \{0\}$ is compact.
    Observe that this class of functions is dense in $L^p(\R^n,X)$. By the smoothness of the
    multiplier $m$, we also have that $T_m f \in \mathcal{S}(\R^n,X)$. Hence, we can apply Lemma~\ref{Lem:11.23}
    and get
    \begin{align*}
        \| T_m f \|_{L^p(\R^n,X)}
            = & \lim_{\eps \to 0} \eps^{n/p'} \Big\| \sum_{k \in \Z^n} \widehat{T_m f}(\eps k) e_k \Big\|_{L^p(\T^n,X)}
            =  \lim_{\eps \to 0} \eps^{n/p'} \Big\| \sum_{k \in \Z^n} m(k) \widehat{f}(\eps k) e_k \Big\|_{L^p(\T^n,X)} \\
            = & \lim_{\eps \to 0} \eps^{n/p'} \Big\| T_{\widetilde{m}}\Big(\sum_{k \in \Z^n} \widehat{f}(\eps k) e_k \Big) \Big\|_{L^p(\T^n,X)}.
    \end{align*}
    Notice that in the second equality we have applied that $m$ is a homogeneous function.
    Now, since the function between parentheses is in $L^p_0(\T^n,X)$, we use hypothesis $(ii)$
    and again Lemma~\ref{Lem:11.23} to conclude that
    \begin{align*}
        \| T_m f \|_{L^p(\R^n,X)}
            \leq & K  \lim_{\eps \to 0} \eps^{n/p'} \sum_{j=1}^N\Big\| T_{\widetilde{m}_j}\Big(\sum_{k \in \Z^n} \widehat{f}(\eps k) e_k \Big) \Big\|_{L^p(\T^n,X)} \\
            = & K  \lim_{\eps \to 0} \sum_{j=1}^N\eps^{n/p'} \Big\| \sum_{k \in \Z^n} \widehat{T_{m_j}f}(\eps k) e_k \Big\|_{L^p(\T^n,X)}
            =  K \sum_{j=1}^N \| T_{m_j}f \|_{L^p(\R^n,X)}.\qedhere
    \end{align*}
\end{proof}

\begin{proof}[Proof of Proposition~\ref{Prop:RnTn}, $(i) \Rightarrow (ii)$]
    Let $f$ be a trigonometric polynomial with $\widehat{f}(0)=0$. This family of functions is
    dense in $L^p_0(\T^n,X)$. We can write $f= \sum_{k \in \Z^n} \widehat{f}(k)e_k=\sum_{k\in I}\widehat{f}(k)e_k$, where
    $I=\{ k \in \Z^n : \widehat{f}(k) \neq 0\}$ is a finite set such that $0 \notin I$.

    We choose the auxiliary functions $\phi(x)=e^{-\pi|x|^2/p}$ and $\psi(x)=e^{-\pi|x|^2/p'}$, which
    clearly satisfy (see, for instance, \cite[Proposition 8.24]{Foll})
    $$1
        =\|\phi\|_{L^p(\R^n)}
        =\|\psi\|_{L^{p'}(\R^n)}
        =\int_{\R^n} \phi(x) \psi(x) dx
        =\int_{\R^n} \widehat{\phi}(\xi) \widecheck{\psi}(\xi) d\xi.$$
    Then, by Lemma~\ref{Lem:11.20}, Hölder's inequality,
    Lemma~\ref{Lem:11.21} and hypothesis $(i)$ we arrive at
    \begin{align*}
        \|T_{\widetilde{m}} f\|_{L^p(\T^n,X)}
            = & \sup_{g} \int_{\T^n}\langle g(t), T_{\widetilde{m}}f(t) \rangle dt
            =   \sup_{g} \lim_{\eps \to 0} \eps^n \int_{\R^n}\langle \psi(\eps x)g(x) , T_m(\phi(\eps \cdot )f)(x)\rangle dx \\
            \leq & \sup_{g} \lim_{\eps \to 0} \eps^{n/p'}  \|\psi(\eps \cdot)g \|_{L^{p'}(\R^n,X^*)} \eps^{n/p} \|T_m(\phi(\eps \cdot )f)\|_{L^{p}(\R^n,X)} \\
            \leq & K\lim_{\eps \to 0}   \sum_{j=1}^N \eps^{n/p} \|T_{m_j}(\phi(\eps \cdot )f)\|_{L^{p}(\R^n,X)},
    \end{align*}
    where the supremum is over all trigonometric polynomials $g : \T^n \to X^*$ such that
    $\|g\|_{L^{p'}(\T^n,X^*)} \leq 1$.

    To finish the reasoning we are going to prove that
    \begin{equation}\label{eq:goal1}
        \lim_{\eps \to 0} \eps^{n/p} \|T_{m_j}(\phi(\eps \cdot )f)\|_{L^{p}(\R^n,X)}
            \leq \|T_{\widetilde{m}_j}f\|_{L^{p}(\T^n,X)}, \quad j=1, \dots, N.
    \end{equation}
    For simplicity, from now on, we just write $m$ instead of $m_j$, $j=1, \dots, N$.

    Take $\theta_0 \in C^\infty_c(\R^n)$ such that $\supp \theta_0 \subset B(0,2)$,
    $\theta_0 \equiv 1$ in $B(0,1)$ and consider
    $$\theta_\ell (\xi)
        = \theta_0 \Big( \frac{\xi}{2^\ell} \Big) - \theta_0 \Big( \frac{\xi}{2^{\ell-1}} \Big), \quad \xi \in \R^n, \ \ell \geq 1.$$
    Note that the function $\theta_\ell$ is supported in the annulus $B(0,2^{\ell+1}) \setminus B(0,2^{\ell-1})$.
    Then, we have the partition of the unity
    $$1= \sum_{\ell \in \N} \theta_\ell (\xi), \quad \xi \in \R^n,$$
    where for each $\xi$ there exists at most two nonzero terms in the above sum
    (see \cite[p. 242]{Ste2} for details). Let also $\Theta_\ell =\theta_{\ell-1}+\theta_{\ell}+\theta_{\ell+1}$
    ($\theta_{-1}=0$), so that $\Theta_\ell$ is supported in $B(0,2^{\ell+2})$ for $\ell=0,1$,
    and in $B(0,2^{\ell+2})\setminus B(0,2^{\ell-2})$ for $\ell\geq 2$;
    and $\Theta_\ell=1$ on the support of $\theta_\ell$.

     For every $\ell \in \N$, we define $\phi_\ell \in \mathcal{S}(\R^n)$
    to be the function such that $\widehat{\phi}_\ell = \widehat{\phi} \ \theta_\ell$.
    Now, we can write
    \begin{align*}
        T_m(\phi(\eps \cdot )f)
            & = \sum_{k \in \Z^n} \widehat{f}(k) T_m(\phi(\eps \cdot )e_k)
            = \sum_{k \in \Z^n} \widehat{f}(k) \Big[ m(\xi) \frac{1}{\eps^n} \widehat{\phi} \Big( \frac{\xi-k}{\eps} \Big)  \Big]^{\vee} \\
            & = \sum_{\ell \in \N} \sum_{k \in \Z^n} \widehat{f}(k) \Big[ m(\xi) \frac{1}{\eps^n} \widehat{\phi}_\ell \Big( \frac{\xi-k}{\eps} \Big)  \Big]^{\vee}
            = \sum_{\ell \in \N} T_m(\phi_\ell(\eps \cdot )f).
    \end{align*}
    Since $\widehat{\phi}_\ell((\cdot - k)/\eps)$ is supported in $B(k,2^{\ell+1}\eps)$, it suggests that we may replace
    $m(\xi)$ by $m(k)$ with the advantage that
    $$ \sum_{\ell \in \N} \sum_{k \in \Z^n} \widehat{f}(k) \Big[ m(k) \widehat{\phi_\ell(\eps \cdot)e_k}(\xi)  \Big]^{\vee}
        = \sum_{\ell \in \N} \phi_\ell(\eps \cdot) T_{\widetilde{m}} f
        = \phi(\eps \cdot) T_{\widetilde{m}} f.$$
    With this motivation in mind, we write
    \begin{align*}
        \lim_{\eps \to 0} \eps^{n/p} \|T_{m}(\phi(\eps \cdot )f)\|_{L^{p}(\R^n,X)}
            \leq &  \lim_{\eps \to 0} \eps^{n/p} \|T_{m}(\phi(\eps \cdot )f) - \phi(\eps \cdot) T_{\widetilde{m}} f\|_{L^{p}(\R^n,X)} \\
            & +  \lim_{\eps \to 0} \eps^{n/p} \|\phi(\eps \cdot) T_{\widetilde{m}} f\|_{L^{p}(\R^n,X)}.
    \end{align*}
    Thus, by Lemma~\ref{Lem:11.21}, to establish \eqref{eq:goal1} we only need to see that
    the first term converges to zero. Even more, it is sufficient to show that
    \begin{equation}\label{eq:goal2}
        \lim_{\eps \to 0} \sum_{\ell \in \N} \eps^{n/p} \|T_{m(\cdot)-m(k)}(\phi_\ell(\eps \cdot )e_k)\|_{L^{p}(\R^n)} = 0, \quad k \in I.
    \end{equation}
    Notice that \eqref{eq:goal2} is a scalar property, so the Banach space $X$ does not play any role hereinafter.

We first observe an easy estimate:
\begin{equation*}
\begin{split}
  \eps^{n/p} \|T_{m(\cdot)-m(k)}(\phi_\ell(\eps \cdot )e_k)\|_{L^{p}(\R^n)}
 &\leq \eps^{n/p} \Big(\|T_m\|_{L^p(\R^n)\longrightarrow L^p(\R^n)}+\|m\|_{L^\infty(\R^n)}\Big)\|\phi_\ell(\eps \cdot )e_k\|_{L^{p}(\R^n)} \\
 &\leq C\|\phi_\ell\|_{L^{p}(\R^n)}, \\
\end{split}
\end{equation*}
where moreover we
    claim that
    \begin{equation}\label{eq:claim2}
        \|\phi_\ell\|_{L^p(\R^n)}
            \leq C 2^{-\rho \ell}, \quad \rho \in \N.
    \end{equation}
    Indeed, fixing some $\nu>n/(2p)$ and any (large) $\mu\in\N$, we have
    \begin{align*}
        \|\phi_\ell\|_{L^p(\R^n)}
            & \leq \|(1+|x|^2)^\nu \phi_\ell\|_{L^\infty(\R^n)} \| (1+|x|^2)^{-\nu} \|_{L^p(\R^n)} \\
            & \leq C \Big\|\Big[(1+ \Delta)^\nu \widehat{\phi_\ell}\Big]^{\vee}\Big\|_{L^\infty(\R^n)}
              \leq C \| (1+ \Delta)^\nu \widehat{\phi_\ell} \|_{L^1(\R^n)}
              \leq C \sum_{|\gamma| \leq 2 \nu } \| \partial^\gamma \widehat{\phi_\ell} \|_{L^1(\R^n)}\\
            &\leq C \sum_{|\gamma| \leq 2 \nu }\sum_{\delta\leq\gamma}\binom{\gamma}{\delta}
                \| \partial^{\delta} \widehat{\phi}\cdot \partial^{\gamma-\delta}\theta_{\ell} \|_{L^1(\R^n)}
            \leq C \sum_{|\delta| \leq 2 \nu }
                \| \partial^{\delta} \widehat{\phi}\|_{L^1(B(0,2^{\ell+1})\setminus B(0,2^{\ell-1}))} \\
            & \leq C 2^{(n-\mu) \ell} \sum_{|\delta| \leq 2 \nu } \| |\xi|^\mu \partial^\delta \widehat{\phi}(\xi) \|_{L^\infty(\R^n)}
             \leq C 2^{(n-\mu) \ell},
    \end{align*}
    since $\phi$ and thus $\widehat\phi$ belong to $\mathcal{S}(\R^n)$.

By dominated convergence, it suffices to show that each term in \eqref{eq:goal2}, for a fixed $\ell\in\N$, tends to zero as $\eps\to 0$. 

    Fix $k \in I$, $\ell \in \N$ and $0<\eps<2^{-\ell-3}$.
    If we define $M_\ell^{\eps,k} = (m(\cdot)-m(k))\Theta_\ell((\cdot-k)/\eps)$, it is clear that
    $$T_{m(\cdot)-m(k)}(\phi_\ell(\eps \cdot )e_k)
        = T_{M_\ell^{\eps,k}}(\phi_\ell(\eps \cdot )e_k),$$
    and
    \begin{align}\label{eq:phiTM}
        \eps^{n/p} \|T_{M_\ell^{\eps,k}}(\phi_\ell(\eps \cdot )e_k)\|_{L^{p}(\R^n)}
            \leq & \| \phi_\ell \|_{L^{p}(\R^n)} \|T_{M_\ell^{\eps,k}}\|_{L^{p}(\R^n) \longrightarrow L^{p}(\R^n)}.
    \end{align}
We already estimate $\| \phi_\ell \|_{L^{p}(\R^n)}$ above, and we now turn to the multiplier norm.
     First of all, it is more convenient to consider the new multiplier
    $$\widetilde{M_\ell^{\eps,k}}(\eta)
        = M_\ell^{\eps,k}(2^\ell \eps \eta + k) = [m(2^\ell \eps \eta + k) - m(k)] \Theta_\ell (2^\ell \eta), \quad \eta \in \R^n,$$
    because the $L^p$-norm of Fourier multipliers is invariant under dilations and translations
    in the multiplier function (see for example \cite[(2.5.14) and (2.5.15)]{Graf}).
    Note that $\supp \Theta_\ell (2^\ell \cdot) \subset B(0,4)$, so that
    $|k+2^\ell\eps\eta|\geq|k|- |\eta| \cdot 2^\ell\eps \geq 1- 4 \cdot 1/8=1/2$,
    and thus $m$ is only evaluated in $B(0,1/2)^c$, where it is $C^\infty$. Also note that $\Theta_\ell(2^\ell\eta)=\Theta_2(2^2\eta)$ for all $\ell\geq 2$.

We first estimate $\widetilde{M_\ell^{\eps,k}}$ pointwise.
By applying the mean value theorem:
    \begin{align*}
        |\widetilde{M_\ell^{\eps,k}}(\eta)|
        &\leq \max_{z\in[k,k+2^\ell\eps\eta]} |\nabla m(z)|\cdot 2^\ell\eps|\eta|\cdot |\Theta_\ell(2^\ell\eta)| \\
            & \leq 2^{\ell+2} \eps \sup_{|z|\geq 1/2} |\nabla m(z)|
            \leq C 2^{\ell} \eps, \quad \eta \in \R^n,
    \end{align*}
    where the last bound follows from the homogeneity of $m$ (see for example \cite[p. 366]{Graf}). (Here
     $[x,y]$ stands for the segment connecting points $x$, $y \in \R^n$ and
    $C$ does not depend on $\eta$, $\ell$ or $\eps$.)

    Next, we estimate the size of the derivatives of $\widetilde{M_\ell^{\eps,k}}$. From the Leibniz rule, 
    we get for every $\gamma \in \N^n \setminus \{0\}$, such that $|\gamma| \leq n+1$,
    \begin{equation*}
\begin{split}
         |\partial^\gamma \widetilde{M_\ell^{\eps,k}}(\eta)|
          &  \leq   \Big\{ |m(2^{\ell}\eps \eta + k) - m(k)|\cdot |(\partial^\gamma) (\Theta_\ell)(2^{\ell}\cdot))(\eta)| \\
                 &  \qquad\qquad+ \sum_{0 \neq \tau \leq \gamma} \binom{\gamma}{\tau} (2^\ell\eps)^{|\tau|} |(\partial^\tau m)(2^{\ell}\eps \eta + k)|
                  \cdot |(\partial^{\gamma - \tau})( \Theta_\ell) (2^{\ell}\cdot))(\eta)|\Big\} \\
           & \leq   C\Big\{ 2^{\ell} \eps+\sum_{0\neq\tau\leq\gamma}(2^\ell\eps)^{|\tau|}\Big\}\leq C 2^\ell\eps, \quad \eta \in \R^n,
\end{split}
\end{equation*}
    for certain constant $C$ independent of $\eta$, $\ell$ and $\eps$. The same bound holds for $|\eta|^{|\gamma|}|\partial^\gamma \widetilde{M_\ell^{\eps,k}}(\eta)|$, since $|\eta|\leq 4$ on the support of $ \widetilde{M_\ell^{\eps,k}}$. Hence, Mihlin's multiplier theorem
    (see for instance \cite[Theorem 5.2.7]{Graf}) implies that
    \begin{equation}\label{eq:Mihlin}
        \| T_{\widetilde{M_\ell^{\eps,k}}} \|_{L^{p}(\R^n) \longrightarrow L^{p}(\R^n)}
            \leq C 2^{\ell} \eps.
    \end{equation}

    Putting together \eqref{eq:claim2}, \eqref{eq:phiTM} and \eqref{eq:Mihlin} we conclude \eqref{eq:goal2}.
\end{proof}

To finish this section we present a transference result from $\T^n$ to $\T^{rn}$. Its idea when $n=1$ goes back to Bourgain \cite{Bou}.
Given an operator $S$ acting on $L^p_0(\T^n,X)$ we define its tensor extension
to $L^p_0(\T^n,L^p(\T^{(\ell-1)n},X)) \simeq L^p(\T^{(\ell-1)n},L^p_0(\T^{n},X))$, $\ell \in \N$, by
$$S_\ell = I_{L^p(\T^{(\ell-1)n})} \otimes S.$$

\begin{Prop}\label{Prop:TnTrn}
    Let $1<p<\infty$, $N,r \in \N$ and $m$, $m_j \in C^\infty(\R^n \setminus \{0\})$, $j=1, \dots, N$,
    be homogeneous functions.  If the following estimate holds
    $$\displaystyle \|T_{\widetilde{m}} f\|_{L^p(\T^n,X)}
                            \leq K \sum_{j=1}^N \| T_{\widetilde{m}_j}f\|_{L^p(\T^n,X)}, \quad f \in L^p_0(\T^n,X),$$
    then,
    $$\displaystyle \Big \| \sum_{\ell=1}^r (T_{\widetilde{m}})_{\ell} f_\ell \Big\|_{L^p(\T^{rn},X)}
                            \leq K \sum_{j=1}^N \Big\| \sum_{\ell=1}^r (T_{\widetilde{m}_j})_\ell f_\ell \Big \|_{L^p(\T^{rn},X)},
                            \quad f_\ell \in L^p_0(\T^n,L^p(\T^{(\ell-1)n},X)).$$
\end{Prop}

\begin{proof}
    By density arguments it is enough to consider trigonometric polynomials $(f_\ell)_{\ell=1}^r$.
    For each $\ell =1, \dots, r$, we write
    \begin{equation*}\label{eq:expansion}
        f_\ell(\bar{t}_{\ell-1},t_\ell)
            = \sum_{\substack{k \in \Z^n \setminus \{0\} \\ |k| \leq B}}
            \sum_{\substack{s \in \Z^{(\ell-1)n} \\ |s| \leq B}}
            a_{s,k}^{(\ell)} e_s(\bar{t}_{\ell-1}) e_k(t_\ell),
    \end{equation*}
    where $\bar{t}_{\ell-1}=(t_1, \dots, t_{\ell-1}) \in \T^{(\ell-1)n}$, $t_\ell \in \T^n$ and
    $a_{s,k}^{(\ell)} \in X$. Notice that it is posible to choose the same $B \in \N$ for all $f_\ell$.
    Then,
    \begin{align*}
        (T_{\widetilde{m}})_{\ell} f_\ell(\bar{t}_{\ell-1},t_\ell)
            & = \sum_{\substack{k \in \Z^n \setminus \{0\} \\ |k| \leq B}}
                \sum_{\substack{s \in \Z^{(\ell-1)n} \\ |s| \leq B}}
                a_{s,k}^{(\ell)} e_s(\bar{t}_{\ell-1}) T_{\widetilde{m}} \Big(e_k(t_\ell)\Big) \\
            & = \sum_{\substack{k \in \Z^n \setminus \{0\} \\ |k| \leq B}}
                \sum_{\substack{s \in \Z^{(\ell-1)n} \\ |s| \leq B}}
                a_{s,k}^{(\ell)} e_s(\bar{t}_{\ell-1}) m(k) e_k(t_\ell), \quad (\bar{t}_{\ell-1},t_\ell) \in \T^{\ell n}.
    \end{align*}

    Fix some $\bar{t}_\ell=(\bar{t}_{\ell-1},t_\ell)=(t_1, \dots, t_\ell) \in \T^{\ell n}$ and take
    $\bar{M}_\ell = (\bar{M}_{\ell-1},M_\ell)=(M_1, \dots, M_\ell) \in (\N \setminus\{0\})^\ell$
    to be chosen below. We introduce some operations between vectors of different lengths,
    $$\bar{M}_\ell \otimes t
        = (\bar{M}_{\ell-1} \otimes t, M_\ell t)= (M_1 t , \dots, M_\ell t) \in (\T^{n})^\ell,
        \quad t \in \T^n,$$
    $$\bar{M}_{\ell-1} \odot s
        = M_1 s_1 + \dots + M_{\ell-1}s_{\ell-1} \in \Z^n, \quad s=(s_1, \dots, s_{\ell-1}) \in (\Z^n)^{\ell-1}.$$
    It is verified that
    $s \cdot (\bar{M}_{\ell-1} \otimes t) = (\bar{M}_{\ell-1} \odot s) \cdot t$,
    and hence $e_s(\bar{M}_{\ell-1} \otimes t)=e_{\bar{M}_{\ell-1} \odot s}(t)$.

    Now, we consider Bourgain's transformation of $f_\ell$, which is given by
    \begin{align*}\label{eq:11.22}
        \widetilde{f}_\ell(t)
            = f_\ell(\bar{t}_{\ell} + \bar{M}_\ell \otimes t)
            = \sum_{\substack{k \in \Z^n \setminus \{0\} \\ |k| \leq B}}
            \sum_{\substack{s \in \Z^{(\ell-1)n} \\ |s| \leq B}}
            a_{s,k}^{(\ell)} e_s(\bar{t}_{\ell-1}) e_k(t_\ell)e_{\bar{M}_{\ell-1} \odot s + M_\ell k}(t),
            \quad t \in \T^n.
    \end{align*}
    The function $\widetilde{(T_{\widetilde{m}})_\ell f_\ell}$ is defined analogously.

    We want to compare $\widetilde{(T_{\widetilde{m}})_\ell f_\ell}$ with $T_{\widetilde{m}} \widetilde{f}_\ell$.
    Observe that both of them are multipliers transforms of $\widetilde{f}_\ell$, where in the first case
    the exponential factor $e_{\bar{M}_{\ell-1} \odot s + M_\ell k}$ is multiplied by $m(k)$; and in the second
    one by $m(\bar{M}_{\ell-1} \odot s + M_\ell k)$. Using the homogeneity of $m$ and the mean value theorem,
    we arrive at
    \begin{align*}
        |m(k) - m(\bar{M}_{\ell-1} \odot s + M_\ell k)|
            & = \Big| m\Big( \frac{k}{|k|} \Big) - m\Big( \frac{\bar{M}_{\ell-1} \odot s}{M_\ell|k|} + \frac{k}{|k|}\Big) \Big| \\
            & \leq \Big| \frac{\bar{M}_{\ell-1} \odot s}{M_\ell|k|} \Big| \sup_z |\nabla m(z)|
            \leq  \frac{|\bar{M}_{\ell-1}|}{|M_\ell|}  \sup_z |\nabla m(z)|,
    \end{align*}
    where the supremum is taken over certain compact set away from the origin. This supremum is
    finite because $m \in C^\infty(\R^n \setminus \{0\})$. If we choose the sequence
    $M_1 < M_2 < \dots$ to be sufficiently rapidly increasing, we can make above difference smaller
    than any preassigned $\eps>0$.

    In conclusion, denoting by $\|g\|_A$ the sum of the $X$-norms of the
    Fourier coefficients of a trigonometric polynomial $g$ (on a torus of any dimension), we have
    deduced that
    $$\Big\|\widetilde{(T_{\widetilde{m}})_\ell f_\ell} - T_{\widetilde{m}} \widetilde{f}_\ell \Big\|_{L^p(\T^n,X)}
        \leq \Big\|\widetilde{(T_{\widetilde{m}})_\ell f_\ell} - T_{\widetilde{m}} \widetilde{f}_\ell \Big\|_{A}
        \leq \eps \|f_\ell\|_A, \quad \ell=1, \dots r,$$
    provided that $M_1 < M_2 < \dots$ is sufficiently rapidly increasing.

    Now, we apply the hypothesis in order to get
    \begin{align*}
        \Big\| \sum_{\ell=1}^r \widetilde{(T_{\widetilde{m}})_\ell f_\ell} \Big\|_{L^p(\T^n,X)}
        & \leq \Big\| T_{\widetilde{m}} \Big( \sum_{\ell=1}^r \widetilde{f}_\ell \Big) \Big\|_{L^p(\T^n,X)}
            + \eps \sum_{\ell=1}^r \|f_\ell\|_A \\
        & \leq K \sum_{j=1}^N \Big \| T_{\widetilde{m}_j} \Big( \sum_{\ell=1}^r \widetilde{f}_\ell \Big) \Big\|_{L^p(\T^n,X)}
            + \eps \sum_{\ell=1}^r \|f_\ell\|_A.
    \end{align*}
    So far, we have taken $L^p$-norms with respect to the variable $t \in \T^n$.
    It is time to take $L^p$-norms with respect to the fixed variables
    $\bar{t}_r=(t_1, \dots, t_r) \in \T^{rn}$,
    \begin{align*}
        & \Big( \int_{\T^{rn}} \int_{\T^n} \Big\| \sum_{\ell=1}^r
            (T_{\widetilde{m}})_\ell f_\ell(\bar{t}_\ell + \bar{M}_\ell \otimes t) \Big\|^p_X dt d\bar{t_r} \Big)^{1/p} \\
        & \qquad \leq K \sum_{j=1}^N \Big( \int_{\T^{rn}} \int_{\T^n} \Big\| \sum_{\ell=1}^r
             T_{\widetilde{m}_j} \widetilde{f}_\ell (\bar{t}_\ell + \bar{M}_\ell \otimes t) \Big\|^p_X dt d\bar{t_r} \Big)^{1/p}
            + \eps \sum_{\ell=1}^r \|f_\ell\|_A.
    \end{align*}
    Finally, exchanging the order of integration we notice that the dependence in $t$ and $\bar{M}_\ell$
    disappears, resulting in
    $$\Big \| \sum_{\ell=1}^r (T_{\widetilde{m}})_{\ell} f_\ell \Big\|_{L^p(\T^{rn},X)}
        \leq K \sum_{j=1}^N \Big\| \sum_{\ell=1}^r (T_{\widetilde{m}_j})_\ell f_\ell \Big \|_{L^p(\T^{rn},X)}
            + \eps \sum_{\ell=1}^r \|f_\ell\|_A,$$
    which finishes the proof of this lemma, once we take $\eps \to 0$.
\end{proof}

\section{Proof of Theorem~\ref{Th:main}} \label{sec:proofTh}

We start this section with an alternative, but equivalent, definition of a UMD Banach space.
Let $(\Omega,d\mu)$ be a probability space. We call $(\eps_\ell d_\ell(\eps_1, \dots, \eps_{\ell-1}))_{\ell \in \N}$
a Paley--Walsh martingale difference sequence when $d_\ell : \R^{\ell-1} \longrightarrow X$, $\ell \geq 2$; $d_1$ is
a constant map in $X$ and $(\eps_\ell)_{\ell \in \N}$ are independent Bernoulli random variables, that is,
$\mu(\{\omega \in \Omega : \eps_\ell(\omega)=\pm 1\})=1/2$, $\ell \in \N$.

The Banach space $X$ has UMD if
for some (equivalently, for every) $1<p<\infty$, there exists a constant $c_p$ such that
\begin{equation}\label{eq:defUMD2}
    \Big\| \sum_{\ell=1}^r \sigma_\ell \eps_\ell d_\ell(\eps_1, \dots, \eps_{\ell-1}) \Big\|_{L^p(\Omega,X)}
        \leq c_p \Big\| \sum_{\ell=1}^r \eps_\ell d_\ell(\eps_1, \dots, \eps_{\ell-1}) \Big\|_{L^p(\Omega,X)},
\end{equation}
for each $r \in \N$, all Walsh-Paley martingale difference sequence $(\eps_\ell d_\ell)_{\ell=1}^r$ and
every $(\sigma_\ell)_{\ell =1}^r \in \{-1,1\}^r$.
It is well-known that $\beta_p(X)= \inf c_p$ (see \cite[p. 12]{Bu4} and \cite{Mau}).

The next lemma reduces checking the condition \eqref{eq:defUMD2} to the following inequality, more convenient  for
our purposes.

\begin{Lem}\label{Lem:reduction}
    Let $1<p<\infty$. Suppose that there exists $C_p>0$ such that the following holds:
    For every $r\in\N$ and a sequence $(\sigma_\ell)_{\ell=1}^r$ in $\{-1,+1\}$, there is a sequence
    $(b_\ell)_{\ell=1}^r$ in $\{-1,+1\}^n$
    such that
    \begin{equation}\label{eq:reduction}
        \Big\| \sum_{\ell=1}^r \sigma_\ell a(b_\ell \cdot t_\ell) \Phi_\ell(t_1, \dots, t_{\ell-1}) \Big\|_{L^p(\T^{rn},X)}
            \leq C_p \Big\| \sum_{\ell=1}^r a(b_\ell \cdot t_\ell) \Phi_\ell(t_1, \dots, t_{\ell-1}) \Big\|_{L^p(\T^{rn},X)},
    \end{equation}
    for every $1$-dimensional trigonometric polynomial $a$ with zero mean value and every trigonometric polynomial
    $\Phi_\ell:\T^{(\ell-1)n}\longrightarrow X$.
    Then $\beta_p(X) \leq C_p$.
\end{Lem}

\begin{proof}
    Let $r \in \N$, let $(\eps_\ell d_\ell)_{\ell=1}^r$ be a given Paley--Walsh martingale difference sequence  and
    $(\sigma_\ell)_{\ell =1}^r \in \{-1,1\}^r$.
    Take also $b_\ell \in \{-1,1\}^n$, $\ell=1, \dots, r$, satisfying \eqref{eq:reduction}.

    It is obvious that the function $\sgn(\theta)$, $\theta \in \T$, is a Bernoulli random variable.
    Moreover, for every $b \in \{-1,1\}^n$, the  function $\sgn( b \cdot t)$, $t \in \T^n$, is also a
    Bernoulli random variable.
    Since $(\eps_1(\omega), \dots, \eps_r(\omega))$ and $(\sgn(b_1 \cdot t_1), \dots, \sgn(b_r \cdot t_r))$ have the same
    distribution, 
    we can write
\begin{equation}\label{eq:equid}
\begin{split}
     &\Big\| \sum_{\ell=1}^r \sigma_\ell' \eps_\ell d_\ell\Big(\eps_1, \dots, \eps_{\ell-1})\Big) \Big\|_{L^p(\Omega,X)}  \\
     &= \Big\| \sum_{\ell=1}^r \sigma_\ell' \sgn(b_\ell \cdot t_\ell) d_\ell\Big(\sgn(b_1 \cdot t_1), \dots, \sgn(b_{\ell-1} \cdot t_{\ell-1})\Big) \Big\|_{L^p(\T^{rn},X)},
\end{split}
\end{equation}
for both $\sigma_\ell'=\sigma_\ell$ and $\sigma_\ell'=1$.

     For a fixed $\delta>0$, we choose a
    $1$-dimensional trigonometric polynomial $a$ with zero mean value, verifying that
    $$\|\sgn(\cdot) - a \|_{L^p(\T)} < \delta,$$
where $\sgn$ is the $1$-periodic extension of the sign on $[-1/2,1/2)$.
Then it easily follows that, for every $\ell=1, \dots, r$,
    $$\|\sgn(b_\ell \cdot t) - a(b_\ell \cdot t) \|_{L^p(\T^n)} < \delta,$$
 and so in particular
 \begin{equation*}
  \|a(b_\ell \cdot t) \|_{L^p(\T^n)} \leq \|\sgn(b_\ell \cdot t)\|_{L^p(\T^n)} +\delta=1+\delta.
\end{equation*}

We also choose trigonometric polynomials $\Phi_\ell:\T^{(\ell-1)n}\longrightarrow X$ such that
\begin{equation*}
   \Big\|\Phi_\ell(t_1,\ldots,t_{\ell-1})-d_\ell\Big(\sgn(b_1 \cdot t_1), \dots, \sgn(b_{\ell-1} \cdot t_{\ell-1})\Big) \Big\|_{L^p(\T^{(\ell-1)n},X)}<\delta,
   \quad \ell=1, \dots, r.
\end{equation*}
It follows that
\begin{equation}\label{eq:approx}
\begin{split}
  &\Big\|a(b_\ell\cdot t_\ell)\Phi_\ell(t_1,\ldots,t_{\ell-1})
  -\sgn(b_\ell\cdot t_{\ell})d_\ell\Big(\sgn(b_1 \cdot t_1), \dots, \sgn(b_{\ell-1} \cdot t_{\ell-1})\Big) \Big\|_{L^p(\T^{\ell n},X)} \\
  &\leq\|a(b_\ell\cdot t_{\ell})\|_{L^p(\T^n)}\Big\|\Phi_\ell(t_1,\ldots,t_{\ell-1})
  -d_\ell\Big(\sgn(b_1 \cdot t_1), \dots, \sgn(b_{\ell-1} \cdot t_{\ell-1})\Big) \Big\|_{L^p(\T^{(\ell-1) n},X)} \\
  &\quad+\|a(b_\ell\cdot t_\ell)-\sgn(b_\ell\cdot t_\ell)\|_{L^p(\T^n)}
  \Big\|d_\ell\Big(\sgn(b_1 \cdot t_1), \dots, \sgn(b_{\ell-1} \cdot t_{\ell-1})\Big) \Big\|_{L^p(\T^{(\ell-1) n},X)} \\
  &\leq (1+\delta)\cdot\delta+\delta\cdot C \leq C\delta, \quad \ell=1, \dots, r,
\end{split}
\end{equation}
where the first $C$ is the maximum of the $L^p$ norms of the given functions $d_\ell(\eps_1,\ldots,\eps_{\ell-1})$.

A combination of \eqref{eq:equid} and \eqref{eq:approx} with the assumption \eqref{eq:reduction} now shows that
\begin{equation*}
\begin{split}
  &\Big\| \sum_{\ell=1}^r \sigma_\ell \eps_\ell d_\ell(\eps_1, \dots, \eps_{\ell-1}) \Big\|_{L^p(\Omega,X)}
  \leq \Big\| \sum_{\ell=1}^r \sigma_\ell a(b_\ell \cdot t_\ell) \Phi_\ell(t_1,\ldots,t_{\ell-1}) \Big\|_{L^p(\T^{rn},X)} +Cr\delta \\
  &\leq C_p \Big\| \sum_{\ell=1}^r  a(b_\ell \cdot t_\ell) \Phi_\ell(t_1,\ldots,t_{\ell-1}) \Big\|_{L^p(\T^{rn},X)} +Cr\delta \\
  &\leq C_p\Big\| \sum_{\ell=1}^r  \eps_\ell d_\ell(\eps_1,\ldots,\eps_{\ell-1}) \Big\|_{L^p(\Omega,X)}+(C_p+1)Cr\delta.
\end{split}
\end{equation*}
    Since $\delta>0$ is arbitrarily small, we conclude that \eqref{eq:defUMD2} holds and $\beta_p(X) \leq C_p$.
\end{proof}

\begin{proof}[Proof of Theorem~\ref{Th:main}]
Let us first make an additional useful reduction. Note that the set $F$ can be neither $\varnothing$ or $\{1,\ldots,n\}$, because for these sets we have $\sum_{\ell\in F}\beta_\ell=\sum_{\ell\in F}\alpha^j_\ell$ (using $|\beta|=|\alpha^j|$ if $F=\{1,\ldots,n\}$), contradicting the fact that these have different parity. Thus, there exists an index $s\in F$ as well as $t\in\{1,\ldots,n\}\setminus F$. By applying the assumption \eqref{eq:estim2} to $\partial_k u$ in place of $u$, we see that \eqref{eq:estim2}  implies a similar estimate with $\beta+e_k$ in place of $\beta$ and $\alpha^j+e_k$ in place of $\alpha^j$. Thus, if $\sum_{\ell\in F}\beta_\ell$ is not already odd, we can make it odd by adding $e_s$ to both $\beta$ and $\alpha^j$; note that this preserves the other assumptions. After, if the new $|\beta|$ is not even, we can make it even by adding $e_t$ to both $\beta$ and $\alpha^j$, which again preserves the other assumptions, including the value of $\sum_{\ell\in F}\b!
 eta_\ell$, since $t\notin F$. Thus, if the assumptions of Theorem~\ref{Th:main} hold for some $\beta,\alpha^j$, they also hold for some (possibly different) $\beta,\alpha^j$ which satisfy in addition that $|\beta|=|\alpha^j|$ is even and $\sum_{\ell\in F}\beta_{\ell}$ is odd (so that $\sum_{\ell\in F}\alpha^j_\ell$ is again even). We henceforth assume this, and proceed to prove the claim that $X$ is UMD with $\beta_p(X)\leq KN$.

So, let $\alpha^j$, $\beta$ and $F$ be as in the
statement of the theorem, with the additional assumptions just made, and $m$, $m_j$ be the multipliers defined in \eqref{eq:mmj}.
The evenness of $|\beta|=|\alpha^j|$ gives the advantage that the multipliers
$m(\xi)=\xi^\beta/|\xi|^{|\beta|}$ and $m_j(\xi)=\xi^{\alpha^j}/|\xi|^{|\alpha^j|}$ are not only homogeneous, but also even, so that $m(\lambda\xi)=m(\xi)$ holds for all $\lambda\in\R\setminus\{0\}$, not just $\lambda>0$.

Fix a 1-dimensional trigonometric polynomial
$$a(\theta)
    = \sum_{\ell \in \Z} \widehat{a}(\ell) e_\ell(\theta), \quad \theta \in \T,$$
with $\widehat{a}(0)=0$, and consider the function $a_b(t)=a(b \cdot t)$, $t \in \T^n$, for certain $b \in \{-1,1\}^n$.
Also, if $J \subset \{1, \dots, n\}$ we define $b_J \in \{-1,1\}^n$ to be the vector
whose $\ell$-th component is equal to $-1$, if $\ell \in J$; or $1$ otherwise. With this notation,
take $a^+=a_{(1, \dots, 1)}$ and $a^-=a_{b_{F}}$.

Since $m(\ell b)=m(b)$ for all $\ell\in\Z\setminus\{0\}$, we find that
\begin{equation}\label{eq:claim}
    T_{\widetilde{m}} a_b(t)
    =T_{\widetilde{m}}\sum_{\ell\in\Z}\widehat{a}(\ell)e_{\ell b}(t)
    =\sum_{\ell\in\Z}\widehat{a}(\ell)m(\ell b)e_{\ell b}(t)
        = m(b)a_b(t), \quad b \in \{-1,1\}^n,
\end{equation}
where $T_{\widetilde{m}}$ the discrete Fourier multiplier defined in \eqref{eq:Tmdiscret}.
Now, taking in mind the assumptions imposed over $F$, we deduce
as a direct consequence of \eqref{eq:claim} that
\begin{equation}\label{eq:pm}
    T_{\widetilde{m}} a^{\pm} = \pm n^{-|\beta|/2} a^{\pm}
    \quad \text{and} \quad
    T_{\widetilde{m}_j} a^{\pm} = n^{-|\beta|/2} a^{\pm}, \ j=1, \dots, N.
\end{equation}

Let $r \in \N$, $\sigma_\ell \in \{-1,1\}$ and
$\Phi_\ell:\T^{(\ell-1)n}\longrightarrow X$ be trigonometric polynomials for $\ell=1, \dots, r$.
Moreover, we define
$\zeta_\ell = a^+$, if $\sigma_\ell=1$; and $\zeta_\ell = a^-$, when $\sigma_\ell=-1$.
As it was commented in the introduction, \eqref{eq:estim2} implies \eqref{eq:estim3}. Then,
Propositions~\ref{Prop:RnTn} and \ref{Prop:TnTrn} together with \eqref{eq:pm}, allows us to write
\begin{align*}
     \Big\| \sum_{\ell=1}^r \sigma_\ell \zeta_\ell \Phi_\ell \Big\|_{L^p(\T^{rn},X)}
     & = n^{|\beta|/2} \Big\| \sum_{\ell=1}^r T_{\widetilde{m}}\zeta_\ell \Phi_\ell \Big\|_{L^p(\T^{rn},X)}
      \leq  K n^{|\beta|/2} \sum_{j=1}^N \Big\| \sum_{\ell=1}^r T_{\widetilde{m}_j}\zeta_\ell \Phi_\ell \Big\|_{L^p(\T^{rn},X)} \\
     & = K N \Big\| \sum_{\ell=1}^r \zeta_\ell \Phi_\ell \Big\|_{L^p(\T^{rn},X)}.
\end{align*}
Finally, an application of Lemma~\ref{Lem:reduction} implies that $\beta_p(X) \leq KN$, and hence the proof of
this theorem is completed.
\end{proof}



\def\cprime{$'$}

\end{document}